\documentclass[12pt]{article}
\usepackage[utf8]{inputenc}
\usepackage[T1]{fontenc}
\usepackage[english]{babel}
\usepackage{parskip}
\usepackage{amsmath}
\usepackage{amssymb}
\usepackage{amsthm}
\usepackage{amsfonts}
\usepackage[hidelinks]{hyperref}
\usepackage{xcolor}
\usepackage{enumitem}
\usepackage{graphicx}
\usepackage{mathtools}
\usepackage{tikz}
\usetikzlibrary{arrows.meta,positioning,calc,backgrounds}
\usepackage{caption}
\usepackage{microtype}

\newtheorem{theorem}{Theorem}[section]
\newtheorem{lemma}[theorem]{Lemma}
\newtheorem{proposition}[theorem]{Proposition}
\newtheorem{corollary}[theorem]{Corollary}

\theoremstyle{definition}
\newtheorem{definition}[theorem]{Definition}

\newtheorem{example}[theorem]{Example}
\newtheorem{remark}[theorem]{Remark}
\theoremstyle{plain}

\tikzset{vertex/.style={circle,draw,fill=white,inner sep=2pt},
 terminal/.style={rectangle,draw,fill=gray!15,inner sep=3pt},
 point/.style={circle,fill,inner sep=1.5pt},
 bridge/.style={dashed,thick},
 transport/.style={-{Stealth[length=2mm]},thick}}

\newcommand{\GH}{d_{\mathrm{GH}}}
\newcommand{\Hd}{d_{\mathrm H}}
\newcommand{\dis}{\operatorname{dis}}
\newcommand{\len}{\operatorname{len}}
\newcommand{\diam}{\operatorname{diam}}
\newcommand{\Lip}{\operatorname{Lip}}
\newcommand{\ulim}{\lim_{\mathcal U}}
\newcommand{\R}{\mathbb{R}}
\newcommand{\N}{\mathbb{N}}

\newcommand{\cC}{\mathcal{C}}
\newcommand{\cM}{\mathcal{M}}
\newcommand{\cK}{\mathcal{K}}
\newcommand{\U}{\mathcal{U}}
\newcommand{\mycite}[2]{\cite[#1]{#2}}

\title{\vspace{-2cm}Steiner Minimal Trees in the Gromov--Hausdorff Space\\
and in Gromov--Hausdorff Clouds}
\author{Ivan Mikhailov}
\date{}

\begin{document}
\maketitle

\begin{abstract}
We prove the existence of Steiner minimal trees spanning a finite family of metric spaces in two cases of the Gromov--Hausdorff theory. In the compact Gromov--Hausdorff space $\cM$, every finite family of nonempty compact metric spaces is spanned by a shortest Steiner minimal tree whose Steiner vertices are again nonempty compact metric spaces. In a fixed Gromov--Hausdorff cloud $\cC$, every finite family of proper, isometrically cobounded spaces admits a finite-length Steiner minimal tree all of whose vertices remain in $\cC$. The compact case uses compact replacement through distance profiles and the Hausdorff hyperspace, and the cloud case uses pointed ultralimits.
\end{abstract}

\tableofcontents

\clearpage

%%%%%%%%%%%%%%%%%%%%%%%%%%%%%%%%%%%%%%%%%%
\section{Introduction}
\markright{\thesection.~Introduction}
%%%%%%%%%%%%%%%%%%%%%%%%%%%%%%%%%%%%%%%%%%

% The Hausdorff and Gromov--Hausdorff distances, their correspondence formula, diameter estimates, realization in a common ambient space, and the compactness of the Hausdorff hyperspace are standard; we use the conventions and results of \cite{BBI,TuzhilinLectures}. General background on Steiner minimal trees and extremal networks may be found in \cite{IvanovTuzhilinBook}.

The \emph{Steiner problem} asks for a shortest connected network spanning a prescribed finite set of terminals in a given metric space; its classical theory is treated in \cite{GilbertPollak,HRW,IvanovTuzhilinBook,IvanovTuzhilinNetworks}.

In the Gromov--Hausdorff space the problem was first studied by \cite{INT}: they proved that every finite family of \emph{finite} metric spaces is spanned by a Steiner minimal tree whose vertices are finite metric spaces, with a cardinality bound on those vertices (see~\cite[Theorem~2.1]{INT}); Further background is given in~\cite{ITMinimalFillings}. The present paper extends this existence statement beyond the finite setting, in two directions.

\begin{theorem}[Compact Steiner minimal trees, Theorem~\ref{main}]
For every finite nonempty family $(X_i)_{i\in I}$ of nonempty compact metric spaces, there is a finite tree-shaped network with these terminals and nonempty compact metric spaces at all Steiner vertices whose length is no greater than that of every finite connected network with the same terminals. For $|I|\ge2$ one may take at most $|I|-2$ Steiner vertices. Consequently, every finite family of points of $\cM$ admits a Steiner minimal tree inside~$\cM$.
\end{theorem}

\begin{theorem}[Steiner minimal trees in one cloud, Theorem~\ref{maincloud}]
Fix a Gromov--Hausdorff cloud~$\cC$. Every nonempty finite family of proper, isometrically cobounded spaces $X_i\in\cC$, $i\in I$, admits a finite-length Steiner minimal tree in $\cC$; that is, the tree is no longer than any connected competitor with the same terminals whose vertices all lie in $\cC$.
\end{theorem}

The compact case is proved by \emph{compact replacement}: correspondences realize the vertex spaces of a tree in one common metric space $W_\varepsilon$, the vector of distances to the compact terminals is a compact family $L_M$ of $1$-Lipschitz profiles, and the compact Hausdorff hyperspace replaces every Steiner space by a compact one without increasing any edge length. 

The cloud case is proved by \emph{pointed ultralimits}: coherently transported base points compactify the vertex spaces, and properness together with isometric coboundedness identifies each terminal ultrapower with the prescribed terminal. In both cases only finitely many tree shapes occur, since every connected network reduces to a tree with at most $|I|-2$ Steiner vertices of degree at least three.

The paper is organized as follows. Section~\ref{sec:defs} fixes notation, and Section~\ref{sec:standard} collects the standard Steiner-tree results. Section~\ref{sec:glue} proves the gluing lemma, Section~\ref{sec:compact} the compact-case theorem, and Section~\ref{sec:cloud} the cloud-case theorem. 

%%%%%%%%%%%%%%%%%%%%%%%%%%%%%%
\subsection*{Statement on the use of AI tools}
%%%%%%%%%%%%%%%%%%%%%%%%%%%%%%

AI tools were used in the preparation of this text, in particular for the literature search, for the \LaTeX{} typesetting, and for drafting and checking parts of the arguments. All mathematical statements were verified by the author, who is solely responsible for the content of the paper and for any remaining errors.

%%%%%%%%%%%%%%%%%%%%%%%%%%%%%%
\subsection*{Acknowledgements}
%%%%%%%%%%%%%%%%%%%%%%%%%%%%%%

The author is grateful to Prof.\ A.\,A.\,Tuzhilin and Prof.\ A.\,O.\,Ivanov for the constant attention to the work, and to the participants of the seminar ``Theory of extremal networks'' led by Prof.\ A.\,O.\,Ivanov and Prof.\ A.\,A.\,Tuzhilin at Lomonosov Moscow State University, and of the monthly online seminar in Gromov--Hausdorff distance geometry, for numerous fruitful discussions.

%%%%%%%%%%%%%%%%%%%%%%%%%%%%%%%%%%%%%%%%%%
\section{Main definitions and preliminary results}\label{sec:defs}
\markright{\thesection.~Main definitions and preliminary results}
%%%%%%%%%%%%%%%%%%%%%%%%%%%%%%%%%%%%%%%%%%

All metric spaces below are nonempty unless the contrary is stated. We fix notation for the standard material on metric spaces, correspondences and the Gromov--Hausdorff distance; the elementary properties are classical and are cited as they are used.

\subsection{Metric spaces}

\begin{definition}[Bounded, totally bounded, compact]
A metric space $(X,d)$ is \emph{bounded} if $\diam X=\sup_{x,x'\in X}d(x,x')<\infty$, and \emph{totally bounded} if for every $r>0$ finitely many open $r$-balls $B(x,r)=\{y:d(x,y)<r\}$ cover it. It is \emph{compact} if it is complete and totally bounded, equivalently if every sequence has a convergent subsequence. The closed ball of radius $r$ about $x$ is $\overline B(x,r)=\{y:d(x,y)\le r\}$.
\end{definition}

The equivalence of compactness with sequential compactness and the elementary properties of closed balls are standard; see~\cite[Chapter~1 and Section~1.6]{BBI} and~\cite{TuzhilinLectures}.

\subsection{Correspondences and the Gromov--Hausdorff distance}

\begin{definition}[Correspondence and distortion]
Let $(X,d_X)$ and $(Y,d_Y)$ be metric spaces. A \emph{correspondence} between them is a relation $R\subseteq X\times Y$ whose two coordinate projections are surjective. Its \emph{distortion} is
\[
 \dis R=\sup\Bigl\{|d_X(x,x')-d_Y(y,y')|:(x,y),(x',y')\in R\Bigr\}\in[0,\infty].
\]
The \emph{Gromov--Hausdorff distance} is
\[
 \GH(X,Y)=\tfrac12\inf_R\dis R,
\]
where $R$ ranges over all correspondences between $X$ and $Y$. For nonempty compact $X,Y$ this is the usual Gromov--Hausdorff distance.
\end{definition}

\begin{lemma}[Basic distance facts, \mycite{Chapter~7}{BBI}]\label{basic}
The Gromov--Hausdorff distance is symmetric and satisfies the triangle inequality. If $\GH(X,Y)<\infty$, then for each $\eta>0$ there is a correspondence with distortion less than $2\GH(X,Y)+\eta$. If $X,Y$ are bounded, then
\[
 |\diam X-\diam Y|\le 2\GH(X,Y).
\]
If $Y$ is bounded and $\GH(X,Y)<\infty$, then $X$ is bounded as well.
\end{lemma}

For nonempty bounded subsets $A,B$ of a metric space $W$, the \emph{Hausdorff distance} is
\[
 \Hd^W(A,B)=\max\Bigl\{\sup_{a\in A}\inf_{b\in B}d_W(a,b),\ \sup_{b\in B}\inf_{a\in A}d_W(a,b)\Bigr\}.
\]

\begin{lemma}[Ambient and correspondence definitions, \mycite{Chapter~7}{BBI}]\label{ambient}
For nonempty metric subspaces $A,B\subseteq W$, one has $\GH(A,B)\le\Hd^W(A,B)$. Conversely, every correspondence $R\subseteq A\times B$ with $\dis R<\infty$ admits, for every $c>\dis R/2$, isometric embeddings of $A$ and $B$ in a common metric space at Hausdorff distance at most $c$. Thus the infimum of ambient Hausdorff distances equals the correspondence definition of $\GH$.
\end{lemma}
The direction from a correspondence to a common ambient realization is supplied explicitly by Lemma~\ref{glue}.

\begin{lemma}[Compact Hausdorff hyperspace, \mycite{Sections~3--4}{TuzhilinLectures}]\label{hyperspace}
If $L$ is a nonempty compact metric space, the space $\cK(L)$ of nonempty compact subsets of $L$, equipped with the Hausdorff distance, is compact. Moreover,
\[
 \bigl|\Hd(A,B)-\Hd(A',B')\bigr|\le\Hd(A,A')+\Hd(B,B')
 \qquad(A,B,A',B'\in\cK(L)).
\]
\end{lemma}

\begin{lemma}[Zero distance for compact spaces, \mycite{Chapter~7}{BBI}]\label{zero}
Nonempty compact metric spaces $A,B$ with $\GH(A,B)=0$ are isometric.
\end{lemma}

%%%%%%%%%%%%%%%%%%%%%%%%%%%%%%%%%%%%%%%%%%
\section{Steiner trees: standard results}\label{sec:standard}
\markright{\thesection.~Steiner trees: standard results}
%%%%%%%%%%%%%%%%%%%%%%%%%%%%%%%%%%%%%%%%%%

This section gathers the standard Steiner-tree material that the paper quotes; the results are classical and are cited rather than reproved.

\begin{definition}[Network, length, Steiner vertex]\label{networkdef}
Fix a finite index set $I$ and terminal spaces $(X_i,d_i)$, $i\in I$. A \emph{network} consists of a finite simple connected graph with vertex set $I\sqcup V$, where $V$ is a finite set of additional, or \emph{Steiner}, vertices, together with a metric space $Y_a$ at every vertex $a\in I\sqcup V$, where $Y_i=X_i$ for $i\in I$. Its \emph{length} is
\[
 \len N=\sum_{\{a,b\}\in E(N)}\GH(Y_a,Y_b)\in[0,\infty].
\]
Different vertices may carry isometric spaces. A \emph{tree} is a connected graph with no cycle, equivalently a connected graph with a unique simple path between any two vertices. A \emph{tree network} is a network whose graph is a tree. A \emph{Steiner minimal tree} for $(X_i)_{i\in I}$ is a tree network $N_*$ with these terminals and
\[
 \len N_*\le\len N
\]
for every connected network $N$ with the same terminals.
\end{definition}

The degree of a vertex is the number of its neighbours; a \emph{leaf} has degree one. A path is \emph{simple} if it has no repeated vertices, and a \emph{cycle} is a closed path with no repetitions except at its starting and ending vertex.

The first standard result is the combinatorial reduction of connected networks to trees with few Steiner vertices. It is independent of the metric spaces placed at the vertices and holds for arbitrary nonnegative edge weights.

\begin{lemma}[Reduction to tree shapes, \mycite{Chapter~1}{HRW}, \mycite{Chapter~1}{IvanovTuzhilinBook}]\label{reduction}
Let $n=|I|\ge2$. Every finite connected network with terminals $(X_i)_{i\in I}$ can be replaced, without increasing its length, by a tree network with the same terminals and at most $n-2$ Steiner vertices. In particular, only finitely many labelled tree shapes need be considered when searching for a Steiner minimal tree.
\end{lemma}
\begin{proof}
For completeness we indicate the argument: take a spanning tree and discard the remaining edges, which does not increase a sum of nonnegative weights; delete every nonterminal leaf; suppress every nonterminal vertex of degree two, replacing the path $a,v,b$ by the edge $\{a,b\}$, whose weight is at most $\GH(Y_a,Y_v)+\GH(Y_v,Y_b)$ by the triangle inequality (Lemma~\ref{basic}). The operations terminate and leave a tree in which every leaf is a terminal and every Steiner vertex has degree at least three. If the reduced tree has $n$ terminal leaves and $s$ Steiner vertices, the handshake identity and $|E|=|V|-1$ give
\[
 n+3s\le\sum_v\deg v=2(n+s-1),
\]
hence $s\le n-2$.
\end{proof}

\begin{figure}[htbp]
\centering
\begin{tikzpicture}[font=\small]
 \node[anchor=west] at (-.5,2.1) {(a) Delete a nonterminal leaf};
 \node[vertex,label=below:{$u$}] (u) at (0,1.2) {};
 \node[vertex,label=below:{$v\notin I$}] (v) at (2,1.2) {};
 \draw (u)--node[above] {$\GH(Y_u,Y_v)$} (v);
 \draw (u)--(-.5,1.6); \draw (u)--(-.5,.8);
 \draw[transport] (3,1.2)--node[above] {delete $v$} (4.5,1.2);
 \node[vertex,label=below:{$u$}] (ur) at (6,1.2) {};
 \draw (ur)--(5.5,1.6); \draw (ur)--(5.5,.8);
 \node[anchor=west] at (-.5,-.1) {(b) Suppress a nonterminal degree-two vertex};
 \node[vertex,label=below:{$a$}] (a) at (0,-1.1) {};
 \node[vertex,label=below:{$v\notin I$}] (mid) at (2,-1.1) {};
 \node[vertex,label=below:{$b$}] (b) at (4,-1.1) {};
 \draw (a)--node[above,font=\scriptsize] {$\GH(Y_a,Y_v)$} (mid);
 \draw (mid)--node[above,font=\scriptsize] {$\GH(Y_v,Y_b)$} (b);
 \draw[transport] (4.6,-1.1)--(5.5,-1.1);
 \node[vertex,label=below:{$a$}] (ar) at (6,-1.1) {};
 \node[vertex,label=below:{$b$}] (br) at (8.2,-1.1) {};
 \draw[thick] (ar)--node[above,font=\scriptsize] {$\GH(Y_a,Y_b)$} (br);
\end{tikzpicture}
\caption{Local reductions of a competitor tree (Lemma~\ref{reduction}).}
\label{fig:reduction}
\end{figure}
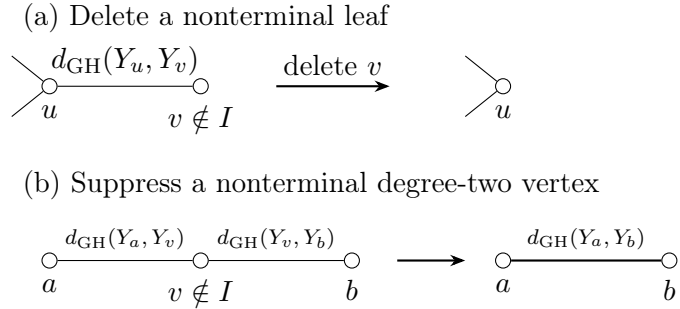

Because of Lemma~\ref{reduction}, the existence proofs below need only handle tree networks with a bounded number of Steiner vertices; the reduction of an arbitrary competitor is then routed through it.

The second standard result is existence: in a finite, or more generally compact, metric space shortest networks exist, and in the Gromov--Hausdorff space the finite-metric-space case is due to Ivanov, Nikolaeva and Tuzhilin.

\begin{theorem}[Existence in the classical setting, \mycite{Chapter~1}{HRW}, \mycite{Chapter~1}{IvanovTuzhilinNetworks}]\label{classicalexist}
Let $X$ be a finite metric space, or more generally a compact metric space, and let $A\subseteq X$ be a finite nonempty set of terminals. Then there is a shortest connected network in $X$ spanning $A$, and it may be taken to be a tree in which every Steiner vertex has degree at least three.
\end{theorem}
\begin{proof}
The compact case follows because there are finitely many tree topologies and, for a fixed topology, the length is continuous on the compact product $X^{V}$. The degree reduction is Lemma~\ref{reduction}.
\end{proof}

\begin{theorem}[Finite metric spaces in the Gromov--Hausdorff space, \mycite{Theorem~2.1}{INT}]\label{INTthm}
Every finite family of finite metric spaces admits a Steiner minimal tree all of whose vertices are finite metric spaces, with a cardinality bound on those vertex spaces.
\end{theorem}

The counting of labelled tree shapes used later is equally standard: on a fixed labelled vertex set there are only finitely many simple graphs, since a graph is a subset of the finitely many unordered pairs of vertices. With at most $n-2$ Steiner vertices available (Lemma~\ref{reduction}), only finitely many tree shapes occur, and each admits finitely many realizations.

%%%%%%%%%%%%%%%%%%%%%%%%%%%%%%%%%%%%%%%%%%
\section{Gluing metric spaces along a tree}\label{sec:glue}
\markright{\thesection.~Gluing metric spaces along a tree}
%%%%%%%%%%%%%%%%%%%%%%%%%%%%%%%%%%%%%%%%%%

To compare the vertex spaces of a network we realize them inside one common metric space, matching the spaces along the edges through the chosen correspondences. The following gluing lemma is the technical core of the compact-case proof.

\begin{lemma}[Gluing one new space]\label{glue}
Let $(W,d_W)$ be a metric space containing an isometric copy of a metric space $(A,d_A)$. Let $(B,d_B)$ be metric, let $R\subseteq A\times B$ be a correspondence, and let $c>0$ satisfy $\dis R\le2c$. Then there is a metric on the disjoint union $W\sqcup B$ restricting to the given metrics and satisfying
\begin{equation}\label{eq:bridge}
 d(a,b)\le c\qquad\bigl((a,b)\in R\bigr).
\end{equation}
Consequently, the Hausdorff distance between $A$ and $B$ in this union is at most $c$.
\end{lemma}
\begin{proof}
Allow movement within $W$ or within $B$ at the original distance, and add a bridge of length $c$ between $a\in A$ and $b\in B$ for each $(a,b)\in R$. Define the distance to be the infimum of the costs of finite chains of allowed movements. Chains exist between any two points because $R$ is a correspondence.

We verify that no original distance is shortened. An excursion from $W$ into $B$ and back has endpoints $a,a'\in A$, bridge partners $b,b'\in B$, and cost at least
\[
 2c+d_B(b,b')\ge d_A(a,a')=d_W(a,a').
\]
Replacing each such excursion by movement within $W$ shows that every chain with endpoints in $W$ has cost at least their $W$-distance. For endpoints in $B$, an excursion through $W$ has cost at least $2c+d_W(a,a')=2c+d_A(a,a')\ge d_B(b,b')$; eliminate these excursions instead. Thus restrictions agree with the original metrics. Every chain between the two disjoint parts contains a bridge and has cost at least $c$, so distinct cross-part points stay at positive distance. Symmetry is immediate, and concatenation of chains followed by approximation of the two infima proves the triangle inequality. Hence this is a metric. Each prescribed bridge has cost $c$, and surjectivity of $R$ proves mutual $c$-closeness.
\end{proof}

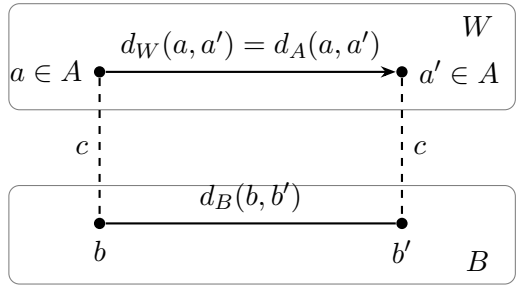
\begin{figure}[htbp]
\centering
\begin{tikzpicture}[font=\small]
 \draw[rounded corners,gray] (-.6,.5) rectangle (6.1,1.9);
 \draw[rounded corners,gray] (-.6,-1.8) rectangle (6.1,-.5);
 \node at (5.6,1.6) {$W$}; \node at (5.6,-1.5) {$B$};
 \node[point,label=left:{$a\in A$}] (a) at (.6,1) {};
 \node[point,label=right:{$a'\in A$}] (ap) at (4.6,1) {};
 \node[point,label=below:{$b$}] (b) at (.6,-1) {};
 \node[point,label=below:{$b'$}] (bp) at (4.6,-1) {};
 \draw[bridge] (a)--node[left] {$c$} (b);
 \draw[bridge] (ap)--node[right] {$c$} (bp);
 \draw[thick] (b)--node[above] {$d_B(b,b')$} (bp);
 \draw[transport] (a)--node[above] {$d_W(a,a')=d_A(a,a')$} (ap);
\end{tikzpicture}
\caption{The correspondence-bridge excursion (Lemma~\ref{glue}).}
\label{fig:excursion}
\end{figure}

\begin{corollary}[Simultaneous tree gluing]\label{treeglue}
Let $H$ be a finite tree and let $Y_a$ be a nonempty metric space at each vertex $a$. Suppose every edge $e=\{a,b\}$ has a correspondence between $Y_a$ and $Y_b$ of distortion at most $2c_e$, where $c_e>0$. Then all $Y_a$ embed isometrically in a common metric space $W$ with
\[
 \Hd^W(Y_a,Y_b)\le c_{\{a,b\}}\qquad\bigl(\{a,b\}\in E(H)\bigr).
\]
For vertices $a,b$, a point $z\in Y_a$ can be transported to a point $w\in Y_b$ with
\[
 d_W(z,w)\le\sum_{e\in[a,b]}c_e,
\]
where $[a,b]$ is the unique tree path.
\end{corollary}
\begin{proof}
Choose a root. Attach each vertex after its parent, applying Lemma~\ref{glue} with $A$ equal to the parent's space in the union already constructed. Existing distances remain unchanged. A finite tree admits this ordering because each nonroot vertex has exactly one parent. To transport $z$, choose successively a partner in each edge correspondence along $[a,b]$; the sum of the bridge lengths bounds the resulting distance.
\end{proof}

%%%%%%%%%%%%%%%%%%%%%%%%%%%%%%%%%%%%%%%%%%
\section{Steiner minimal trees in the compact Gromov--Hausdorff space}\label{sec:compact}
\markright{\thesection.~Steiner minimal trees in the compact Gromov--Hausdorff space}
%%%%%%%%%%%%%%%%%%%%%%%%%%%%%%%%%%%%%%%%%%

The \emph{compact Gromov--Hausdorff space} $\cM$ consists of the isometry classes of nonempty compact metric spaces, with the distance $\GH$. We now prove that every finite family of points of $\cM$ is spanned by a Steiner minimal tree inside $\cM$. The essential new ingredient is \emph{compact replacement}: arbitrary Steiner spaces may be replaced by compact ones without increasing any edge length.

\subsection{A compact reservoir of distance profiles}

\begin{lemma}[Compact profile space, the Arzel\`a--Ascoli theorem \mycite{Sections~3.5--3.7}{TuzhilinLectures}]\label{profilespace}
Let $(X_i,d_i)$, $i\in I$, be finitely many nonempty compact metric spaces, and let $M\ge0$. The space
\[
 L_M=\prod_{i\in I}\Bigl\{f:X_i\to[0,M]:\ \bigl|f(x)-f(x')\bigr|\le d_i(x,x')\ \text{for all }x,x'\in X_i\Bigr\},
\]
with metric $d_L(f,g)=\max_{i\in I}\sup_{x\in X_i}|f_i(x)-g_i(x)|$, is compact whenever $I\ne\varnothing$.
\end{lemma}

\subsection{Distance profiles and compact replacement}

\begin{lemma}[Profiles of a glued tree]\label{profile}
Let $H$ be a finite tree with nonempty terminal set $I$, nonempty compact terminal spaces $X_i$, and nonempty spaces $Y_a$ at all vertices. Suppose its total $\GH$ edge length is at most $C<\infty$. For every $0<\varepsilon\le1$ there are compact sets $A_a^\varepsilon$ in a single compact profile space $L_M$, independent of $\varepsilon$, such that
\begin{enumerate}[label=(\roman*)]
\item $A_i^\varepsilon$ is isometric to $X_i$ for each terminal $i\in I$;
\item $\Hd(A_a^\varepsilon,A_b^\varepsilon)\le\GH(Y_a,Y_b)+\varepsilon$ for every edge $\{a,b\}$.
\end{enumerate}
One may take $M=D+3C+|E(H)|$, where $D=\max_{i\in I}\diam X_i$.
\end{lemma}
\begin{proof}
Put $g_e=\GH(Y_a,Y_b)$ for $e=\{a,b\}$. Choose edge correspondences of distortion at most $2(g_e+\varepsilon)$ by Lemma~\ref{basic}, and glue the tree with bridge lengths $c_e=g_e+\varepsilon$ using Corollary~\ref{treeglue}. Write $W_\varepsilon$ for the resulting common ambient metric space.

Fix a terminal $i_0\in I$. The diameter estimate uses the \emph{exact} $\GH$ edge lengths, not the approximate correspondences. For adjacent vertices $u,v$, Lemma~\ref{basic} gives
\[
 \diam Y_u\le\diam Y_v+2\GH(Y_u,Y_v)=\diam Y_v+2g_{\{u,v\}}.
\]
Starting at the bounded terminal $X_{i_0}$ and summing along the path from $a$ to $i_0$ gives $\diam Y_a\le D+2C$, and shows in particular that every vertex space is bounded. By contrast, actual transport inside $W_\varepsilon$ uses bridges of length $c_e=g_e+\varepsilon$, so the transport estimate below does contain $\varepsilon$.

For any vertex $a$, any $z\in Y_a$, any terminal $i\in I$ and any $x\in X_i$, transport $z$ to a point $w\in X_i$ along the tree as in Figure~\ref{fig:profile-path}. Write $[a,i]=(v_0,\ldots,v_k)$ with $v_0=a$, $v_k=i$ and $e_j=\{v_{j-1},v_j\}$. Choose transported points $z_j\in Y_{v_j}$ with $z_0=z$, $z_k=w$ and $(z_{j-1},z_j)$ in the correspondence on $e_j$ for every $1\le j\le k$.
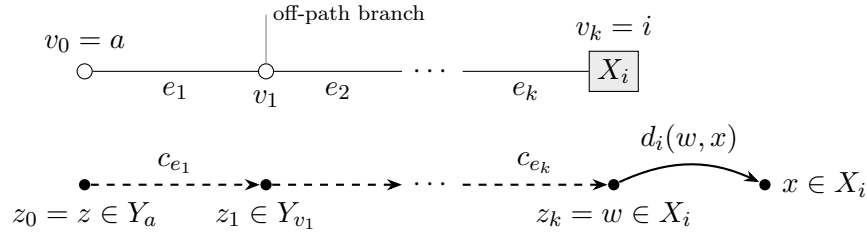
\begin{figure}[htbp]
\centering
\begin{tikzpicture}[font=\small]
 \node[vertex,label=above:{$v_0=a$}] (a) at (0,1.5) {};
 \node[vertex,label=below:{$v_1$}] (b) at (2.4,1.5) {};
 \node (dots) at (4.6,1.5) {$\cdots$};
 \node[terminal,label=above:{$v_k=i$}] (i) at (7,1.5) {$X_i$};
 \draw (a)--node[below] {$e_1$} (b);
 \draw (b)--node[below] {$e_2$} (dots);
 \draw (dots)--node[below] {$e_k$} (i);
 \draw[gray] (b)--(2.4,2.25); \node[font=\scriptsize] at (3.5,2.25) {off-path branch};
 \node[point,label=below:{$z_0=z\in Y_a$}] (z) at (0,0) {};
 \node[point,label=below:{$z_1\in Y_{v_1}$}] (z1) at (2.4,0) {};
 \node (pdots) at (4.6,0) {$\cdots$};
 \node[point,label=below:{$z_k=w\in X_i$}] (w) at (7,0) {};
 \draw[bridge,transport] (z)--node[above] {$c_{e_1}$} (z1);
 \draw[bridge,transport] (z1)--(pdots);
 \draw[bridge,transport] (pdots)--node[above] {$c_{e_k}$} (w);
 \node[point,label=right:{$x\in X_i$}] (x) at (9,0) {};
 \draw[transport] (w) to[bend left=25] node[above] {$d_i(w,x)$} (x);
\end{tikzpicture}
\caption{Transport along the terminal path (Lemma~\ref{profile}).}
\label{fig:profile-path}
\end{figure}
Then
\[
 d_{W_\varepsilon}(z,x)\le\sum_{e\in[a,i]}c_e+d_i(w,x)\le C+|E(H)|+D\le M.
\]
The larger value $M=D+3C+|E(H)|$ also bounds distances between arbitrary vertex-space points; either bound suffices here.

For every vertex $a\in I\sqcup V$ the \emph{distance-profile map} has domain $Y_a$ and codomain the compact function space $L_M$ of Lemma~\ref{profilespace}:
\[
 \Phi_a^\varepsilon:Y_a\longrightarrow L_M,\qquad
 \Phi_a^\varepsilon(z)=\bigl(x\mapsto d_{W_\varepsilon}(z,x)\bigr)_{i\in I}\quad(z\in Y_a).
\]
Explicitly, its codomain is $L_M=\prod_{i\in I}\Lip_1\bigl(X_i,[0,M]\bigr)$, where $\Lip_1\bigl(X_i,[0,M]\bigr)$ denotes the set of all $1$-Lipschitz functions from $X_i$ to $[0,M]$. The $i$-th coordinate of $\Phi_a^\varepsilon(z)$ is the function
\[
 \bigl(\Phi_a^\varepsilon(z)\bigr)_i:X_i\longrightarrow[0,M],\qquad
 \bigl(\Phi_a^\varepsilon(z)\bigr)_i(x)=d_{W_\varepsilon}(z,x)\qquad(x\in X_i),
\]
which is $1$-Lipschitz because
\[
 \bigl|\bigl(\Phi_a^\varepsilon(z)\bigr)_i(x)-\bigl(\Phi_a^\varepsilon(z)\bigr)_i(x')\bigr|
 \le d_i(x,x')\qquad(i\in I,\ x,x'\in X_i,\ z\in Y_a)
\]
by the triangle inequality. Hence $\Phi_a^\varepsilon$ is well defined. Set $A_a^\varepsilon=\Phi_a^\varepsilon(Y_a)$, a compact subset of $L_M$ because $Y_a$ is bounded and $\Phi_a^\varepsilon$ is continuous.

\begin{remark}[Profiles and the Kuratowski embedding]\label{kuratowski}
At a terminal $i$ the $i$-th coordinate of a profile is the function $x\mapsto d_{W_\varepsilon}(z,x)$ on $X_i$; since $z,x\in X_i$ and $d_{W_\varepsilon}$ extends $d_i$ there, this is the map $z\mapsto d_i(z,\cdot)$. The latter is exactly the Kuratowski--Fr\'echet embedding of $X_i$ into the space $C(X_i)$ of continuous functions with the supremum metric; see \mycite{Theorem~2.25}{TuzhilinLectures}. Thus the terminal profile map $\Phi_i^\varepsilon$ is the Kuratowski embedding of $X_i$ into the $i$-th factor of $L_M$, and the profile map $\Phi_a^\varepsilon$ extends this embedding to all vertex spaces of the tree at once.
\end{remark}

For a terminal $i$ the map $\Phi_i^\varepsilon$ is an isometry of $X_i$ onto $A_i^\varepsilon$: its $i$-th coordinate is the identity, since $z,x\in X_i$ and $d_{W_\varepsilon}(z,x)=d_i(z,x)$, while its other coordinates are determined by these. Thus $A_i^\varepsilon$ is isometric to $X_i$.

Finally, let $\{a,b\}$ be an edge and let $z\in Y_a$. Choose a bridge partner $z'\in Y_b$ with $(z,z')\in R_{\{a,b\}}$; then $d_{W_\varepsilon}(z,z')\le c_{\{a,b\}}$. For every terminal $i$ and every $x\in X_i$,
\[
 \bigl|\bigl(\Phi_a^\varepsilon(z)\bigr)_i(x)-\bigl(\Phi_b^\varepsilon(z')\bigr)_i(x)\bigr|
 =\bigl|d_{W_\varepsilon}(z,x)-d_{W_\varepsilon}(z',x)\bigr|
 \le d_{W_\varepsilon}(z,z')\le c_{\{a,b\}}.
\]
Hence $d_L\bigl(\Phi_a^\varepsilon(z),\Phi_b^\varepsilon(z')\bigr)\le c_{\{a,b\}}$, and, taking the maximum over $z\in Y_a$,
\[
 \Hd(A_a^\varepsilon,A_b^\varepsilon)\le c_{\{a,b\}}=g_{\{a,b\}}+\varepsilon.
\]
This is (ii), and the proof is complete.
\end{proof}

\begin{theorem}[Compact replacement]\label{replacement}
Let $H$ be a finite tree with nonempty compact terminal spaces $X_i$, $i\in I$, and nonempty spaces $Y_a$ at all vertices, of total $\GH$ edge length at most $C<\infty$. Then there are nonempty compact spaces $K_a$ with $K_i$ isometric to $X_i$ for every terminal $i$ and
\[
 \GH(K_a,K_b)\le\GH(Y_a,Y_b)
\]
for every edge $\{a,b\}$. The terminal spaces may consequently be kept equal to the prescribed $X_i$ without increasing any edge length.
\end{theorem}
\begin{proof}
Apply Lemma~\ref{profile} with $\varepsilon=1/j$, $j\ge1$. All sets $A_a^{1/j}$ lie in the fixed compact space $L_M$. By Lemma~\ref{hyperspace}, extracting successive subsequences at the finitely many vertices, arrange $A_a^{1/j}\to K_a$ in Hausdorff distance for every vertex $a$. Every $K_a$ is nonempty and compact. Continuity of the Hausdorff distance gives $\Hd(K_a,K_b)\le\GH(Y_a,Y_b)$, and Lemma~\ref{ambient} converts this into the required $\GH$ bound. At a terminal $i$, since $A_i^{1/j}$ is isometric to $X_i$,
\[
 \GH(X_i,K_i)=\GH(A_i^{1/j},K_i)\le\Hd(A_i^{1/j},K_i)\longrightarrow0,
\]
so Lemma~\ref{zero} identifies $K_i$ isometrically with $X_i$. Replacing a vertex space by an isometric copy preserves all its $\GH$ distances, so the prescribed terminals can be restored exactly.
\end{proof}

\subsection{Attainment for a fixed tree shape}

\begin{proposition}[Fixed-tree minimum]\label{fixed}
Fix a finite tree $H$ on $I\sqcup V$ with $I\ne\varnothing$ and nonempty compact terminal spaces $X_i$. The infimum of the length over all nonempty metric spaces placed at the Steiner vertices $V$ is attained by nonempty compact spaces.
\end{proposition}
\begin{proof}
A finite-length competitor exists: put a one-point space at each Steiner vertex. All terminal spaces are bounded, and the full product correspondence between two bounded nonempty spaces has finite distortion, so the infimum $\lambda$ is finite and nonnegative. For each $j\ge1$ choose a realization $Y^j$ with length $l_j<\lambda+1/j$; then $\lambda\le l_j\to\lambda$, and all lengths are at most $C=\lambda+1$.

Apply the construction of Lemma~\ref{profile} to $Y^j$ with $\varepsilon=1/j$. Its profile space $L_M$ is the same for every $j$, because the terminal spaces, the bound $C$ and the tree shape are fixed. Let $A_a^j\in\cK(L_M)$ be the closed profile images. Extract a common subsequence with $A_a^j\to K_a$ for every vertex, using Lemma~\ref{hyperspace}. The terminal argument of Theorem~\ref{replacement} shows that $K_i$ is isometric to $X_i$. For every edge $e=\{a,b\}$,
\[
 \GH(K_a,K_b)\le\Hd(K_a,K_b)=\lim_j\Hd(A_a^j,A_b^j).
\]
Summing these finitely many inequalities and using the profile bounds gives
\[
 \sum_{\{a,b\}\in E(H)}\GH(K_a,K_b)
 \le\lim_j\sum_{\{a,b\}\in E(H)}\Hd(A_a^j,A_b^j)
 \le\lim_j\Bigl(l_j+\frac{|E(H)|}{j}\Bigr)=\lambda.
\]
Restore the terminals to $X_i$. The definition of $\lambda$ gives the opposite inequality for this realization, proving attainment. No individual convergence of the original edge lengths is required.
\end{proof}

\subsection{The compact Steiner minimal tree}

\begin{theorem}[Compact Steiner minimal trees]\label{main}
For every finite nonempty family $(X_i)_{i\in I}$ of nonempty compact metric spaces there is a finite tree-shaped network with these terminals and nonempty compact metric spaces at all Steiner vertices whose length is no greater than that of every finite connected network with the same terminals, even if the latter uses arbitrary metric spaces at its Steiner vertices. For $|I|\ge2$ one may take at most $|I|-2$ Steiner vertices.
\end{theorem}
\begin{proof}
If $|I|=1$ the terminal alone is a tree of length zero, which is minimal. Suppose $n=|I|\ge2$. By Lemma~\ref{reduction}, every finite connected competitor reduces, without increasing its length, to a tree with at most $n-2$ Steiner vertices. Labelling the possible Steiner vertices by $\{1,\ldots,m\}$, $m\le n-2$, yields only finitely many labelled tree shapes, since a graph is a subset of the finitely many unordered pairs of a finite vertex set.

Proposition~\ref{fixed} supplies a compact minimum for each shape. Choose the least of these finitely many attained lengths and call its realizing network $N_*$; its length is finite. Any infinite-length competitor is automatically no shorter than $N_*$. If $N$ is a finite-length competitor, then all its spaces are nonempty: a correspondence between an empty and a nonempty space does not exist, so their $\GH$ distance is infinite, and connectivity to a nonempty terminal rules out an empty vertex of a finite-length network. Lemma~\ref{reduction} transforms $N$ into a no-longer tree with at most $n-2$ Steiner vertices, which relabels to one of the considered shapes; its length is at least the minimum for that shape, hence at least $\len N_*$. Therefore $\len N_*\le\len N$ for every competitor.
\end{proof}

\begin{corollary}[A tree inside $\cM$]
Every finite family of points of the compact Gromov--Hausdorff space $\cM$ admits a Steiner minimal tree inside $\cM$.
\end{corollary}
\begin{proof}
Choose a nonempty compact representative $X_i$ of each terminal point and apply Theorem~\ref{main}; then replace every vertex space by its isometry class. Lemmas~\ref{ambient} and~\ref{zero} identify the correspondence distance with the metric of $\cM$, so this passage preserves every edge length and the total length. Conversely a finite network in $\cM$ is represented by finitely many nonempty compact spaces with the same length, so the comparison with all competitors transfers to $\cM$.
\end{proof}

\begin{remark}[Degenerate families and zero-length edges]
If there are no terminals, a single vertex carrying a one-point space is a tree of length zero and is minimal. Repeated terminal isometry classes are allowed: vertices are labelled even when their images in $\cM$ coincide, so zero-length edges are permitted. If networks are required to have distinct image vertices, identify coincident vertices, discard loops and repeated edges, and take a spanning tree; connectivity is preserved and the length does not increase. A mixed family of empty and nonempty terminals has no finite-length connected network and is not a terminal family in the classical space $\cM$.
\end{remark}

The finite-metric-space predecessor of Theorem~\ref{main} is due to \cite{INT}, recalled in Theorem~\ref{INTthm}. The present theorem extends it from finite to arbitrary compact terminals, and the next section replaces compactness by membership in a single Gromov--Hausdorff cloud.

%%%%%%%%%%%%%%%%%%%%%%%%%%%%%%%%%%%%%%%%%%
\section{Steiner minimal trees in a Gromov--Hausdorff cloud}\label{sec:cloud}
\markright{\thesection.~Steiner minimal trees in a Gromov--Hausdorff cloud}
%%%%%%%%%%%%%%%%%%%%%%%%%%%%%%%%%%%%%%%%%%

We now study Steiner minimal trees inside a single cloud. Subsection~\ref{clouds} introduces Gromov--Hausdorff clouds and their elementary properties, and Subsection~\ref{propercobounded} those of proper, isometrically cobounded spaces. We then fix one cloud $\cC$ and require every terminal, Steiner vertex and competitor below to lie in $\cC$, so that all $\GH$ distances are finite; the compactness device is a pointed ultralimit, which preserves the geometry of a proper, isometrically cobounded space even when the chosen base points escape to infinity.

\subsection{Gromov--Hausdorff clouds}\label{clouds}

\begin{definition}[Cloud]\label{clouddef}
For a metric space $X$ its \emph{Gromov--Hausdorff cloud} is
\[
 \cC_X=\{Y:\GH(X,Y)<\infty\},
\]
considered up to zero $\GH$ distance. By the triangle inequality, finite $\GH$ distance is an equivalence relation, so two clouds are either equal or disjoint.
\end{definition}

\begin{lemma}[Finite paths remain in one cloud]\label{cloudpath}
Let $Y_0,\ldots,Y_m$ be metric spaces such that $\GH(Y_{j-1},Y_j)<\infty$ for $1\le j\le m$. Then
\[
 \GH(Y_0,Y_m)\le\sum_{j=1}^{m}\GH(Y_{j-1},Y_j)<\infty.
\]
In particular every vertex of a connected finite-length $\GH$ network belongs to the cloud of each terminal.
\end{lemma}
\begin{proof}
Apply the triangle inequality of Lemma~\ref{basic} successively along the path.
\end{proof}

All nonempty bounded metric spaces lie in one cloud: if $X,Y$ are bounded, the full relation $X\times Y$ is a correspondence with $\dis(X\times Y)\le\max\{\diam X,\diam Y\}$, whence $\GH(X,Y)\le\tfrac12\max\{\diam X,\diam Y\}<\infty$. Conversely, by Lemma~\ref{basic}, a space at finite $\GH$ distance from a bounded space is itself bounded, so that cloud consists exactly of the nonempty bounded metric spaces. 

\subsection{Proper and isometrically cobounded spaces}\label{propercobounded}

\begin{definition}[Properness]\label{proper}
A metric space $(X,d)$ is \emph{proper} if every closed ball $\overline B(x,r)$ is compact.
\end{definition}

\begin{definition}[Isometric coboundedness]\label{cobounded}
A metric space $(X,d)$ is \emph{isometrically cobounded} if there exist $x_0\in X$ and $C\ge0$ such that for every $y\in X$ there is a surjective isometry $g:X\to X$ with
\[
 d\bigl(g(y),x_0\bigr)\le C.
\]
Equivalently, the isometry group of $X$ has a bounded orbit space: every point can be moved into the fixed ball $\overline B(x_0,C)$.
\end{definition}

\begin{example}\label{examples-cobounded}
The spaces $\R$, $\R^n$ and $\mathbb Z^n$, and more generally every proper normed group, are proper and isometrically cobounded: translations move a chosen point to the origin. If $K$ is compact, then $\R^n\times K$ with the Manhattan  distance has the same two properties. Every nonempty compact metric space is proper and isometrically cobounded: fix $x_0$ and take $C=\diam X$ with the identity isometry. Every normed additive group is isometrically cobounded, but an infinite-dimensional normed space needs not be proper, so the two hypotheses must not be conflated.
\end{example}

These two properties play complementary roles: properness supplies compact balls in which bounded sequences have ultrafilter limits, while coboundedness supplies isometries that return an escaping base sequence to a fixed bounded region. Their meaning is made precise in Theorem~\ref{terminalident} and Remark~\ref{whyhyp}. The fact that a bounded proper metric space is compact, and the elementary properties of closed balls, are standard; see~\cite[Chapter~1 and Section~1.6]{BBI,TuzhilinLectures}.

\subsection{Ultrafilters and pointed ultralimits}

\begin{definition}[Nonprincipal ultrafilter]
A \emph{filter} on $\N$ is a family $\U$ of subsets of $\N$, regarded as large sets, that contains $\N$, excludes $\varnothing$, is closed under finite intersections, and is upward closed. It is an \emph{ultrafilter} if for every $A\subseteq\N$ exactly one of $A$ and $\N\setminus A$ belongs to $\U$. It is \emph{nonprincipal} if it contains no finite set, equivalently if it contains every cofinite set. We fix one nonprincipal ultrafilter~$\U$.
\end{definition}

A \emph{filter} on a set $J$ is a family $\mathcal F\subseteq\mathcal P(J)$ such that $J\in\mathcal F$, $\varnothing\notin\mathcal F$, $A,B\in\mathcal F$ implies $A\cap B\in\mathcal F$, and $A\in\mathcal F$, $A\subseteq B$ implies $B\in\mathcal F$. A filter is \emph{proper} when it does not contain $\varnothing$. An \emph{ultrafilter} is a filter $\mathcal U$ such that, for every $A\subseteq J$, exactly one of $A$ and $J\setminus A$ belongs to $\mathcal U$. It is \emph{principal} if it consists of all sets containing some fixed point $j_0\in J$; otherwise it is nonprincipal. On $\N$, a nonprincipal ultrafilter contains every cofinite set and no finite set.

The \emph{ultrafilter lemma} states that every proper filter is contained in an ultrafilter. Applying it to the cofinite filter on $\N$ yields the nonprincipal ultrafilter $\U$ fixed above. If $f:J\to Y$ is a map and $\mathcal U$ is an ultrafilter on $J$, the \emph{ultrafilter limit} of $f$ is a point $y\in Y$ such that, for every neighborhood $V$ of $y$,
\[
 f^{-1}(V)\in\mathcal U.
\]
Such a limit is unique whenever $Y$ is Hausdorff, and it exists for every map into a compact Hausdorff space. In particular, every bounded real sequence $(r_k)$ has a unique ultrafilter limit. We denote it by~$\ulim r_k.$

Thus, by definition, $\lim_{\mathcal U} r_k=L$ means that for every neighborhood $V$ of $L$ in $\R$, the set $\{k\in\N:r_k\in V\}$ belongs to $\mathcal U$. We will use the usual functoriality of ultralimits under continuous maps, as well as the compactness principle just stated; see~\cite[Chapter~I.5]{BH} for futher details.

\begin{definition}[Pointed ultralimit]\label{pulimdef}
Let $(Y_k,d_k,o_k)$ be pointed metric spaces. Define the bounded pointed product
\[
 \operatorname{PUProd}(Y_k,o_k)=\Bigl\{z=(z_k)_{k\in\N}\in\prod_{k\in\N}Y_k:\ \sup_k d_k(z_k,o_k)<\infty\Bigr\}.
\]
For $z=(z_k)_k$ and $z'=(z'_k)_k$ in this product put
\[
 \rho(z,z')=\ulim d_k(z_k,z'_k).
\]
The coordinate distances are bounded because both sequences stay uniformly close to $(o_k)_k$, and $\rho$ is a pseudometric. The \emph{pointed ultralimit}
\[
 \operatorname{PULim}_{\U}(Y_k,o_k)
\]
is the metric separation quotient by the relation $\rho(z,z')=0$. The class of $z$ is written $[z]$.
\end{definition}

\begin{lemma}[Correspondences pass to pointed ultralimits]\label{pucorr}
For every $k\in\N$ let $(Y_k,o_k)$ and $(Z_k,p_k)$ be pointed metric spaces, and let $R_k\subseteq Y_k\times Z_k$ be a correspondence with $(o_k,p_k)\in R_k$. Suppose
\[
 \dis R_k\le c_k,\qquad \sup_k c_k<\infty,\qquad c=\ulim c_k.
\]
Then
\[
 \GH\Bigl(\operatorname{PULim}_{\U}(Y_k,o_k),\operatorname{PULim}_{\U}(Z_k,p_k)\Bigr)\le\frac c2.
\]
\end{lemma}
\begin{proof}
Write $Y_\omega$ and $Z_\omega$ for the two pointed ultralimits and define $R_\omega\subseteq Y_\omega\times Z_\omega$ by
\[
 ([y],[z])\in R_\omega\iff
 \begin{aligned}
 &\text{there are representatives }y=(y_k)_k\text{ and }z=(z_k)_k\\
 &\text{with }(y_k,z_k)\in R_k\text{ for every }k.
 \end{aligned}
\]
We first prove surjectivity onto $Y_\omega$. Given an admissible $y=(y_k)_k$, choose $z_k\in Z_k$ with $(y_k,z_k)\in R_k$. Since $(o_k,p_k)\in R_k$,
\[
 d_{Z_k}(z_k,p_k)\le d_{Y_k}(y_k,o_k)+c_k,
\]
whose right-hand side is uniformly bounded, so $z=(z_k)_k$ is admissible and $([y],[z])\in R_\omega$. The same argument with the two factors reversed proves surjectivity onto $Z_\omega$.

If $([y],[z]),([y'],[z'])\in R_\omega$, choose coordinatewise related representatives. Then
\[
 \bigl|d_{Y_k}(y_k,y'_k)-d_{Z_k}(z_k,z'_k)\bigr|\le c_k.
\]
Taking ultrafilter limits and using continuity of the absolute value gives
\[
 \bigl|d_{Y_\omega}([y],[y'])-d_{Z_\omega}([z],[z'])\bigr|\le c.
\]
Hence $\dis R_\omega\le c$, and the correspondence formula for $\GH$ proves the claim. The same estimate shows that the relation is unaffected by changing representatives by zero pseudodistance.
\end{proof}

\begin{theorem}[Terminal identification]\label{terminalident}
Let $(X,d)$ be proper and isometrically cobounded. For any sequence of base points $o_k\in X$,
\[
 \GH\bigl(\operatorname{PULim}_{\U}(X,o_k),X\bigr)=0.
\]
In fact the pointed ultralimit is isometric to $X$.
\end{theorem}
\begin{proof}
Choose $x_0\in X$ and $C\ge0$ as in Definition~\ref{cobounded}. For each $k$ choose a surjective isometry $g_k:X\to X$ with $d(g_k(o_k),x_0)\le C$. Let $X_\omega=\operatorname{PULim}_{\U}(X,o_k)$.

Take an admissible sequence $z=(z_k)_k$, say $\sup_k d(z_k,o_k)\le M_z$. Then
\[
 d\bigl(g_k(z_k),x_0\bigr)\le d\bigl(g_k(z_k),g_k(o_k)\bigr)+d\bigl(g_k(o_k),x_0\bigr)\le M_z+C,
\]
so all $g_k(z_k)$ lie in the compact ball $\overline B(x_0,M_z+C)$ by properness. Their $\U$-limit exists; denote it by $F([z])\in X$.

The map is well defined: if $[z]=[z']$ in $X_\omega$, then $\ulim d(g_k(z_k),g_k(z'_k))=\ulim d(z_k,z'_k)=0$, so the two compact-space ultralimits coincide. For arbitrary admissible $z,z'$, continuity of the distance function on the compact ball containing both transformed sequences gives
\[
 d\bigl(F([z]),F([z'])\bigr)=\ulim d\bigl(g_k(z_k),g_k(z'_k)\bigr)=\ulim d(z_k,z'_k)=d_{X_\omega}([z],[z']).
\]
Thus $F:X_\omega\to X$ is an isometric embedding.

It is surjective: given $x\in X$, put $z_k=g_k^{-1}(x)$. Then
\[
 d(z_k,o_k)=d\bigl(x,g_k(o_k)\bigr)\le d(x,x_0)+C,
\]
so $(z_k)_k$ is admissible and the transformed sequence $g_k(z_k)$ is constantly $x$; hence $F([z])=x$. Therefore $F$ is an isometry, which implies zero $\GH$ distance.
\end{proof}

\begin{remark}[Why both hypotheses matter]\label{whyhyp}
Properness supplies compact balls in which the transformed bounded sequences have ultrafilter limits. Coboundedness supplies the isometries $g_k$ that move an arbitrary escaping base sequence back into one fixed bounded region. Without properness a bounded ultrapower can be larger than the original space; without coboundedness the pointed limit can be a genuine ``view from infinity''.
\end{remark}

\subsection{Minimization on a fixed tree in a cloud}

Fix a cloud $\cC$, a finite tree $H$ with vertex set $I\sqcup V$, $I\ne\varnothing$, and proper isometrically cobounded terminals $X_i\in\cC$, $i\in I$. A \emph{realization} of $H$ in $\cC$ assigns a space $Y_a\in\cC$ to every Steiner vertex $a\in V$ and fixes $Y_i=X_i$ at the terminals.

\begin{proposition}[Minimum on a fixed tree]\label{fixedtree}
The finite real-valued length functional on the realizations of $H$ in $\cC$ attains its infimum.
\end{proposition}
\begin{proof}
There is a realization in $\cC$: choose a root terminal $i_0\in I$ and put $Y_a=X_{i_0}$ at every Steiner vertex. As all terminals lie in one cloud, every edge length is finite, so the infimum $\lambda$ is a finite nonnegative real number. Choose a minimizing sequence of realizations $N_k=(H,(Y_a^k)_a)$ with
\begin{equation}\label{eq:minseq}
 \lambda\le\len N_k\le\lambda+\frac1{k+1}.
\end{equation}
For an edge $e=\{a,b\}$ set $g_{k,e}=\GH(Y_a^k,Y_b^k)$; by~\eqref{eq:minseq}, $0\le g_{k,e}\le\lambda+1$. Choose a correspondence $R_{k,e}\subseteq Y_a^k\times Y_b^k$ with
\begin{equation}\label{eq:edgecorr}
 \dis R_{k,e}\le2g_{k,e}+\frac1{k+1},
\end{equation}
and use the inverse relation for the opposite orientation.

\paragraph{Coherent base points.} Choose a root terminal $i_0\in I$ and a point $x_0\in X_{i_0}$, and root the tree at $i_0$. At the root put $o_{i_0}^k=x_0$. Recursively, if $b$ is a child of $a$, use surjectivity of $R_{k,\{a,b\}}$ to choose $o_b^k\in Y_b^k$ with $(o_a^k,o_b^k)\in R_{k,\{a,b\}}$. Thus the base points at the two ends of every edge are related by its correspondence, and $o_a^k$ is obtained by transporting $x_0$ along the unique root path from $i_0$ to $a$; see Figure~\ref{fig:roottransport}.

\begin{figure}[htbp]
\centering
\begin{tikzpicture}[font=\small]
 \node[terminal] (r) at (0,0) {$i_0:X_{i_0}$};
 \node[vertex] (v1) at (2.6,0) {$a_1$};
 \node[vertex] (v2) at (5.2,0) {$a_2$};
 \node[vertex] (a) at (7.8,0) {$a$};
 \node[vertex] (off) at (5.2,1.6) {$b$};
 \draw (r)--node[below] {$e_1$} (v1);
 \draw (v1)--node[below] {$e_2$} (v2);
 \draw (v2)--node[below] {$e_m$} (a);
 \draw (v2)--(off);
 \node[point,label=below:{$x_0=o_{i_0}^k$}] (p0) at (0,-1.4) {};
 \node[point,label=below:{$o_{a_1}^k$}] (p1) at (2.6,-1.4) {};
 \node[point,label=below:{$o_{a_2}^k$}] (p2) at (5.2,-1.4) {};
 \node[point,label=below:{$o_a^k$}] (pa) at (7.8,-1.4) {};
 \draw[transport] (p0)--(p1);
 \draw[transport] (p1)--(p2);
 \draw[transport] (p2)--(pa);
\end{tikzpicture}
\caption{Coherent base-point transport along the root path.}
\label{fig:roottransport}
\end{figure}
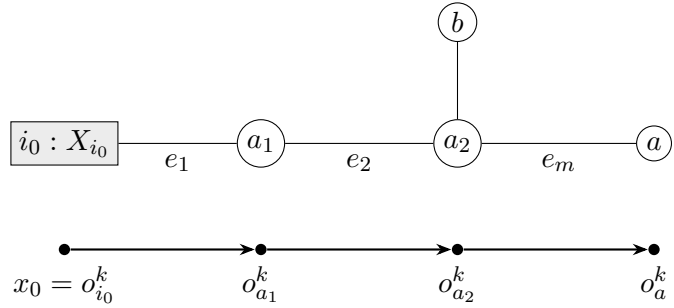

\paragraph{Vertex limits.} For every vertex $a\in I\sqcup V$ define
\[
 Z_a=\operatorname{PULim}_{\U}\bigl(Y_a^k,o_a^k\bigr).
\]
For every edge $e=\{a,b\}$ the numbers $g_{k,e}$ lie in the compact interval $[0,\lambda+1]$, so $\gamma_e=\ulim g_{k,e}$ exists. Apply Lemma~\ref{pucorr} to $R_{k,e}$ and~\eqref{eq:edgecorr}: the base points are related, the distortions are uniformly bounded, and their ultrafilter limit is $2\gamma_e$. Hence
\begin{equation}\label{eq:limit-edge}
 \GH(Z_a,Z_b)\le\gamma_e.
\end{equation}
At a terminal $i\in I$ the coordinate space is constantly $X_i$, though its base point $o_i^k$ may vary; Theorem~\ref{terminalident} gives an isometry $Z_i\cong X_i$. Replace $Z_i$ by the prescribed copy $X_i$; this does not change any incident $\GH$ distance.

For every vertex $a$, follow the finite path in $H$ from $a$ to a terminal. Inequality~\eqref{eq:limit-edge} gives finite $\GH$ distance along every edge, so Lemma~\ref{cloudpath} shows $Z_a\in\cC$. Hence the limit assignment is an admissible realization in the fixed cloud.

The resulting realization $N_\omega$ of $H$ has length
\[
 \len N_\omega\le\sum_{e\in E(H)}\gamma_e=\ulim\sum_{e\in E(H)}g_{k,e}=\ulim\len N_k=\lambda.
\]
The middle equality uses that $E(H)$ is finite, so ultrafilter limits commute with this finite sum. Since $\lambda$ is the infimum, the reverse inequality $\lambda\le\len N_\omega$ is automatic, so $\len N_\omega=\lambda$.
\end{proof}

\subsection{Existence in one cloud}

\begin{theorem}[Steiner minimal tree in one cloud]\label{maincloud}
Fix a Gromov--Hausdorff cloud $\cC$. Every nonempty finite family of proper, isometrically cobounded spaces $X_i\in\cC$, $i\in I$, admits a finite-length Steiner minimal tree in $\cC$.
\end{theorem}
\begin{proof}
If $|I|=1$ the single terminal is a tree of length zero. Assume $|I|\ge2$. By Lemma~\ref{reduction}, every connected competitor in $\cC$ can be replaced, without increasing its length, by a tree network in $\cC$ with at most $2|I|+1$ Steiner vertices, hence it suffices to consider finitely many labelled tree shapes. Proposition~\ref{fixedtree} supplies a minimizing realization in $\cC$ for every shape. Select a shape whose attained minimum is the least among these finitely many real numbers and call the resulting tree $N_*$. Any connected competitor reduces to one of the considered shapes, and its length is at least the minimum for that shape, which is at least $\len N_*$. Hence $N_*$ is a finite-length Steiner minimal tree in $\cC$.
\end{proof}

\end{document}